\documentclass[12pt]{article}
\usepackage{amsfonts}
\usepackage{latexsym}
\usepackage{cite}
\usepackage{amsmath,amsfonts,latexsym,amssymb}
\usepackage[mathscr]{eucal}
\usepackage{cases}
\usepackage{amsthm}

\usepackage[bf,small]{caption2}
\usepackage{float}
\usepackage{graphicx}
\usepackage{amsmath}
\usepackage{amssymb}
\usepackage[all]{xy}

\newtheorem{theorem}{theorem}[section]
\newtheorem{thm}[theorem]{Theorem}
\newtheorem{lem}[theorem]{Lemma}

\newtheorem{prop}[theorem]{Proposition}

\newtheorem{alg}[theorem]{Algorithm}

\newtheorem{defn}[theorem]{Definition}

\begin{document}

\title{\textbf{Lifting automorphisms along abelian regular coverings of graphs}}
\author{\Large Haimiao Chen
\footnote{Email: \emph{chenhm@math.pku.edu.cn}}\\
\normalsize \em{Department of mathematics, Peking University, Beijing, China}}
\date{}
\maketitle

\begin{abstract}
This article proposes an effective criterion for lifting automorphisms along regular coverings of graphs, with the covering
transformation group being any finite abelian group.
\end{abstract}

\emph{key words:} abelian regular covering, lifting isomorphism, criterion.


\section{Introduction}

First let us recall some terminologies in graph theory. For more one can refer to \cite{Crite, Voltage, Graph, Lift, Invar}.

We assume that all graphs are finite, connected, simple (that is, having no loops or parallel edges) and undirected. More precisely, a graph is viewed as a finite one-dimensional simplicial complex (see \cite{AT}, p.104) with orientation forgotten. For a graph $\Gamma$, denote an edge connecting the vertices $u,v$ by $\{u,v\}$; each edge $\{u,v\}$ gives rise to a pair of opposite arcs $(u,v)$ and $(v,u)$. Use $V(\Gamma), E(\Gamma), A(\Gamma)$ to denote the set of vertices, edges, arcs of $\Gamma$, respectively, and use $\textrm{Aut}(\Gamma)$ for the automorphism group.

A \emph{covering} of a graph $\pi:\tilde{\Gamma}\rightarrow\Gamma$ is a simplicial covering map of one-dimensional simplicial complexes. It is called \emph{regular} (or an \emph{$A$-covering}) if there is a subgroup $A$ of the covering transformation group $K$ that acts regularly on each fiber. Moreover, $A=K$ if $\Gamma$ is connected.

Let $A$ be a finite group. An \emph{$A$-voltage assignment} on $\Gamma$ is a function $\phi:A(\Gamma)\rightarrow A$ such that $\phi(u,v)=\phi(v,u)^{-1}$, for all $(u,v)\in A(\Gamma)$; the values of $\phi$ are called \textit{voltages}. The graph $\tilde{\Gamma}=\Gamma\times_{\phi}A$ derived from $\phi$ is defined by $V(\tilde{\Gamma})=V(\Gamma)\times A$, and $E(\tilde{\Gamma})=\{\{(u,g),(v,g\phi(u,v))\}\colon\{u,v\}\in E(\Gamma),g\in A\}$. The projection onto the first coordinate $\pi:\tilde{\Gamma}=\Gamma\times_{\phi}A\rightarrow\Gamma$ defines an $A$-covering.

Given a spanning tree $T$ of $\Gamma$, a \emph{cotree} edge is one not belonging to $T$. A voltage assignment $\phi$ is called \emph{$T$-reduced} if the voltage associated to each tree arc is the identity. It was shown in \cite{Voltage} that every regular covering of $\Gamma$ is isomorphic to a derived graph arising from a $T$-reduced voltage assignment with respect to an arbitrary spanning tree $T$.
The voltage assignment $\phi$ naturally extends to walks in $\Gamma$. For any walk $W$, let $\phi(W)$ denote the voltage of $W$.

An automorphism $\alpha\in\textrm{Aut}(\Gamma)$ is \emph{lifted} to $\tilde{\alpha}\in\textrm{Aut}(\tilde{\Gamma})$ if $\pi\tilde{\alpha}=\alpha\pi$.

In this paper we consider the following decision problem:
Given a graph $\Gamma$, a regular covering $\Gamma\times_{\phi}A$ with $A$ a finite abelian group, and an automorphism $\alpha$ of $\Gamma$, decide whether $\alpha$ can be lifted.

The following proposition of \cite{Lift} plays a key role:
\begin{prop} \label{prop1}
Let $\tilde{\Gamma}=\Gamma\times_{\phi}A$ be a regular covering. Then an automorphism $\alpha$ of $\Gamma$ can be lifted to an automorphism of $\tilde{\Gamma}$ if and only if, for each closed walk $W$ in $\Gamma$, one has $\phi(\alpha(W))=0$ if $\phi(W)=0$.
\end{prop}

Recently, much attention has been paid to the automorphism lifting problem. The motivation was to construct highly transitive graphs; moreover, the lifting conditions are frequently needed in classifying symmetries of certain infinite families of graphs. In \cite{Crite}, a linear criteria is proposed for the case when $A$ is elementary abelian, and applied to classify arc-transitive $A$-coverings of the Petersen graph. In \cite{Invar} the authors classify all vertex-transitive elementary abelian coverings of the Petersen graph; the method, developed in \cite{Elem}, reduces the problem to that of finding invariant subspaces of a group representation. See also \cite{K5, Pappus}, and the references therein.

To the automorphism lifting problem for arbitrary finite abelian groups, we contribute a criterion based on diagonalization of matrices over the ring
$\mathbb{Z}/p^{k}\mathbb{Z}$. The criterion is easy to run in that one only needs to compute some matrices derived from input data, and check whether the entries of the
resulting matrices satisfy certain conditions.

Properties of matrices over $\mathbb{Z}/p^{k}\mathbb{Z}$ that are relevant in the above context, are discussed in Section 2. This is applied in Section 3 in order to derive a criterion for lifting automorphisms. In Section 4, a practical algorithm is presented and illustrated on an example.

\section{On matrices over $\mathbb{Z}/p^{k}\mathbb{Z}$}

We start with some notational conventions. Operations in abelian groups are written additively. For a matrix $C$, denote its $(i,j)$-entry by $C_{i,j}$, and we often write $C=(C_{i,j})$.
For a ring $\mathcal{R}$, let $\mathcal{M}_{m\times n}(\mathcal{R})$ be the set of all $m\times n$-matrices with entries in $\mathcal{R}$. Let $\mathcal{M}_{n}(\mathcal{R})=\mathcal{M}_{n\times n}(\mathcal{R})$, and let $\textrm{GL}(n,\mathcal{R})$ be the set of invertible ones.

Let $p$ be a prime number, and $R=\mathbb{Z}/p^{k}\mathbb{Z}$.
We can speak about the divisibility of an element $\lambda\in R$ by $p^{r},0\leqslant r<k$, without ambiguity. Note that $\lambda$ is invertible if and only if it is not divisible by $p$.

\begin{defn}
\rm For $0\neq\lambda\in R$, the \emph{$p$-degree} of $\lambda$, denoted $d_{p}(\lambda)$, is the largest integer $r,0\leqslant r<k$, such that $p^{r}|\lambda$. By convention set $d_{p}(0)=k$.
\end{defn}
\begin{defn}
\rm A nonzero matrix $X=(X_{i,j})\in\mathcal{M}_{m\times n}(R)$ is \emph{in normal form} if $X_{i,j}=\delta_{i,j}\cdot p^{r_{i}}$ for some integers $r_{1},\cdots,r_{m}$, with $0\leqslant r_{1}\leqslant\cdots\leqslant r_{l}<k=r_{l+1}=\cdots =r_{m}$ for some $l\leqslant \min\{m,n\}$.
\end{defn}
\begin{lem}
For each $0\neq X\in\mathcal{M}_{m\times n}(R)$, there exist $Q\in\emph{GL}(m,R),T\in\emph{GL}(n,R)$ such that $QXT$ is in normal form.
\end{lem}
\begin{proof}
Choose an entry $X_{i_{1},j_{1}}\neq 0$ with smallest $p$-degree, and suppose $X_{i_{1},j_{1}}=p^{r_{1}}\cdot\chi$ with $\chi$ invertible. Interchange the $i_{1}$-th row of $X$ with the first row, and the $j_{1}$-th column with the first column, and divide the first row by $\chi$. The matrix obtained has $(1,1)$-entry $p^{r_{1}}$, dividing all entries. Then perform row-transformations to eliminate the $(i,1)$-entries for $1<i\leqslant m$. Denote the new matrix by $X^{(1)}$.

Do the same thing to the $(m-1)\times(n-1)$ down-right minor of $X^{(1)}$, and then go on. At each step, we can take row transformations to eliminate elements of a submatrix in a column below the main diagonal, and rearrange the diagonal elements. At the $l$-th step, for some $l\leqslant\min\{m,n\}$, we get a matrix $X^{(l)}$ such that the $i$-th diagonal entry is $p^{r_{i}}$ for $1\leqslant i\leqslant m$, with $r_{1}\leqslant\cdots\leqslant r_{m}$, and the elements under the main diagonal all vanish.

Finally, perform column transformations to $X^{(l)}$, to eliminate all the ``off-diagonal" entries. Then the resulting matrix is in normal form.
\end{proof}

\section{The criterion for lifting automorphisms}

From now on we make the followings assumptions regarding a connected graph $\Gamma$ and an $A$-voltage assignment on $\Gamma$.
\begin{enumerate}
\item Let $A=\prod\limits_{\gamma=1}^{g}A_{\gamma}$ denote an abelian group, with $A_{\gamma}=\prod\limits_{\eta=1}^{n_{\gamma}}\mathbb{Z}/p_{\gamma}^{k(\gamma,\eta)}\mathbb{Z}$, where the prime numbers $p_{1},\cdots,p_{g}$ are distinct, and $k(\gamma,1)\leqslant\cdots\leqslant k(\gamma,n_{\gamma})$ for each $\gamma$.
\item For each edge $e\in E(\Gamma)$, one of its arcs is chosen and denoted by $h(e)$.
\item A spanning tree $T$ and a vertex $v_{0}\in V(T)$ are chosen.
For $v_{1},v_{2}\in V(\Gamma)=V(T)$, let $W(v_{1},v_{2})$ denote the unique reduced walk in $T$ from $v_{1}$ to $v_{2}$.
Denote by $e_{1},\cdots,e_{m}$ the cotree edges. For each $e_{i}$ denote by $u_{i}$ and $w_{i}$ the head and the tail of $h(e_{i})$, respectively. Let $L_{i}$ be the unique closed reduced walk based at $v_{0}$ and containing $h(e_{i})$, that is, $L_{i}=W(v_{0},u_{i})\cup h(e_{i})\cup W(w_{i},v_{0})$.
\item Let $\phi$ be a $T$-reduced voltage assignment on $\Gamma$ relative to $T$ and $v_{0}$, so $\phi(h(e))=0$ for all $e\in E(T)$.
\end{enumerate}
\vspace{3mm}

It is well-known that the first homology group $H_{1}(\Gamma;\mathbb{Z})$ is the free abelian group on the set $\{L_{1},\cdots,L_{m}\}$ (see \cite{AT}, p. 43).
The voltage assignment $\phi$ induces a group homomorphism
\begin{align}
\phi^{\ast}:H_{1}(\Gamma;\mathbb{Z})\rightarrow A,
\end{align}
which is surjective by the assumption of connectedness.

For $1\leqslant\gamma\leqslant g$, let $R_{\gamma}=\mathbb{Z}/p_{\gamma}^{k(\gamma,n_{\gamma})}\mathbb{Z}$, considered as a ring. Let $R=\mathbb{Z}/r\mathbb{Z}$, with $r=\prod\limits_{\gamma=1}^{g}p_{\gamma}^{k(\gamma,n_{\gamma})}$.
There are obvious projections $q_{\gamma}:R\twoheadrightarrow R_{\gamma}$ and injections $\iota_{\gamma}:R_{\gamma}\hookrightarrow R$. The projections $q_{\gamma}$ define the canonical isomorphism $q:R\rightarrow\prod\limits_{\gamma=1}^{g}R_{\gamma}$. Abusing the notation, we denote the induced projections $R^{m}\twoheadrightarrow R_{\gamma}^{m}$, injections $R_{\gamma}^{m}\hookrightarrow R^{m}$ and the isomorphism $R^{m}\cong\prod\limits_{\gamma=1}^{g}R_{\gamma}^{m}$ also by $q_{\gamma}, \iota_{\gamma}$ and $q$, respectively.

Note that for all $1\leqslant\eta\leqslant n_{\gamma}$ there is a monomorphism
\begin{align}
\mathbb{Z}/p_{\gamma}^{k(\gamma,\eta)}\mathbb{Z}\hookrightarrow R_{\gamma},\hspace{5mm}
\lambda\mapsto p_{\gamma}^{k(\gamma,n_{\gamma})-k(\gamma,\eta)}\cdot\lambda. \label{eq}
\end{align}

For each $\gamma$, let $\sigma_{\gamma}:A\twoheadrightarrow A_{\gamma}$ and $\tau_{\gamma}:A_{\gamma}\hookrightarrow A$ denote the canonical projection and injection, respectively.

The $p_{\gamma}$ parts of 
$\sigma_{\gamma}(\phi(L_{i}))$ are interpreted as elements of $R_{\gamma}^{n_{\gamma}}$ via (\ref{eq}),
and placed in rows in order to form a nonzero matrix $B_{\gamma}$.

Each automorphism $\alpha\in\textrm{Aut}(\Gamma)$ induces an automorphism
\begin{align}
\alpha_{\ast}:H_{1}(\Gamma;\mathbb{Z})\rightarrow H_{1}(\Gamma;\mathbb{Z});
\end{align}
it is represented by a matrix $S^{\alpha}\in\textrm{GL}(m,\mathbb{Z})$ given by
\begin{align}
\alpha_{\ast}(L_{i})=\sum\limits_{j=1}^{m}(S^{\alpha})_{i,j}L_{j}.
\end{align}

By the monomorphism (\ref{eq}), the voltage group is extended and hence the covering graph is also extended. However, the answer to whether $\alpha$ can be lifted does not change, since by Proposition \ref{prop1}, the condition for lifting $\alpha$ is $\ker\phi^{\ast}=\ker(\phi^{\ast}\circ\alpha_{\ast})$.

In our notations, graph automorphisms are composed on the left, but when representing the action of the first homology group, matrices are composed on the right,
acting on row vectors.

\begin{defn}
\rm For a ring $\mathcal{R}$, and a matrix $X\in\mathcal{M}_{m\times n}(\mathcal{R})$, the \emph{zero set of $X$ in $\mathcal{R}^{m}$}, denoted $Z(X)$, is the set of the row vectors $a=(a_{1},\cdots,a_{m})\in \mathcal{R}^{m}$ satisfying $aX=0$, or explicitly, $\sum\limits_{i=1}^{m}a_{i}X_{i,j}=0, 1\leqslant j\leqslant n$.
\end{defn}

\begin{prop} \label{prop}
The automorphism $\alpha$ can be lifted if and only if
\begin{align}
Z(S^{\alpha}B_{\gamma})=Z(B_{\gamma})\ \text{in\ }R_{\gamma}^{m}, \hspace{5mm} 1\leqslant\gamma\leqslant g,
\end{align}
where $S^{\alpha}$ is regarded as in $\emph{GL}(m,R_{\gamma})$ via the quotient map $\mathbb{Z}\twoheadrightarrow R_{\gamma}$.
\end{prop}
\begin{proof}
For $1\leqslant\gamma\leqslant g$, put $\phi_{\gamma}^{\ast}=\sigma_{\gamma}\circ\phi^{\ast}$.

Identify $H_{1}(\Gamma;\mathbb{Z})$ with $\mathbb{Z}^{m}$. For $a=(a_{1},\cdots,a_{m})\in\mathbb{Z}^{m}$, clearly $a$ belongs to $\ker\phi^{\ast}$ whenever $r|a_{i}, 1\leqslant i\leqslant m$. Hence there is an induced map $\overline{\phi^{\ast}}:(\mathbb{Z}/r\mathbb{Z})^{m}\rightarrow A$. Similarly there are induced maps $\overline{\phi_{\gamma}^{\ast}}:R_{\gamma}^{m}\rightarrow A_{\gamma}$, $1\leqslant\gamma\leqslant g$.

We have the following commuting diagrams
\begin{align*}
\xymatrix{
R^{m}\ar[d]_{\overline{\phi^{\ast}}}\ar[r]^{q_{\gamma}} &R_{\gamma}^{m}\ar[d]^{\overline{\phi_{\gamma}^{\ast}}}\\
A\ar[r]^{\sigma_{\gamma}} &A_{\gamma}
}
\hspace{15mm}\xymatrix{
R^{m}\ar[d]_{\overline{\phi^{\ast}}} &R_{\gamma}^{m}\ar[d]^{\overline{\phi^{\ast}_{\gamma}}}\ar[l]_{\iota_{\gamma}}\\
A &A_{\gamma}\ar[l]_{\tau_{\gamma}}
}
\end{align*}
from which it follows that $q_{\gamma}(\ker\overline{\phi^{\ast}})\subseteq\ker\overline{\phi^{\ast}_{\gamma}}$, and $\iota_{\gamma}(\ker\overline{\phi^{\ast}_{\gamma}})\subseteq\ker\overline{\phi^{\ast}}$.
Hence the isomorphism $q$ sends $\ker\overline{\phi^{\ast}}$ onto $\prod\limits_{\gamma=1}^{g}\ker\overline{\phi^{\ast}_{\gamma}}$.

Replacing $\phi^{\ast}$ with $\phi^{\ast}\circ\alpha_{\ast}$, the same argument leads to the conclusion that, the isomorphism $q$ sends $\ker(\overline{\phi^{\ast}\circ\alpha_{\ast}})$ onto $\prod\limits_{\gamma=1}^{g}\ker(\overline{\phi^{\ast}_{\gamma}\circ\alpha_{\ast}})$.

Thus $\ker\phi^{\ast}=\ker(\phi^{\ast}\circ\alpha_{\ast})$ if and only if $\ker\overline{\phi^{\ast}}=\ker(\overline{\phi^{\ast}\circ\alpha_{\ast}})$, which is equivalent to $\ker\overline{\phi^{\ast}_{\gamma}}=\ker(\overline{\phi^{\ast}_{\gamma}\circ\alpha_{\ast}})$ for all $\gamma$.

Since by Proposition \ref{prop1}, $\alpha$ can be lifted if and only if $\ker\phi^{\ast}=\ker(\phi^{\ast}\circ\alpha_{\ast})$, and by definition, $\ker\overline{\phi^{\ast}_{\gamma}}=Z(B_{\gamma})$ and $\ker(\overline{\phi^{\ast}_{\gamma}\circ\alpha_{\ast}})=Z(S^{\alpha}B_{\gamma})$, the proposition is established.
\end{proof}

\vspace{4mm}

Now we characterize the condition $Z(S^{\alpha}B_{\gamma})=Z(B_{\gamma})$ more concretely.

Fix $\gamma\in\{1,\cdots,g\}$. Simplify notations by $p=p_{\gamma}$, $n=n_{\gamma}$, $k=k(\gamma,n_{\gamma})$, $B=B_{\gamma}$, $R'=R_{\gamma}$.

For a matrix $X\in\mathcal{M}_{l\times m}(R')$, let $\langle X\rangle$ denote the subgroup of $R'^{m}$ generated by the row vectors of $X$.

Since $B\in\mathcal{M}_{m\times n}(R')$, by Lemma 2.3 one can choose matrices $Q\in\textrm{GL}(m,R')$ and $T\in\textrm{GL}(n,R')$ so that $QBT=B^{0}$ is in normal form; let $(B^{0})_{i,j}=\delta_{i,j}\cdot p^{s_{i}}$. Recall that $s_{i}=k$ for all $i>l$, for some $l\leqslant\min\{m,n\}$.

Suppose $i_{0}$ is the smallest $i$ with $s_{i}>0$. It is easy to see that $Z(B^{0})=\langle P\rangle$, where
$P\in\mathcal{M}_{(m-i_{0}+1)\times m}(R')$ with $P_{i,j}=\delta_{i+i_{0}-1,j}\cdot p^{k-s_{j}}$.

\begin{lem} \label{lem}
$Z(S^{\alpha}B)=Z(B)$ if and only if $d_{p}((QS^{\alpha}Q^{-1})_{i,j})\geqslant s_{i}-s_{j}$ for all $1\leqslant j<i\leqslant m$.
\end{lem}
\begin{proof}
Since for all $a\in R^{m}$, $aB=0$ if and only if $aQ^{-1}B^{0}=0$, we have $Z(B)=\langle PQ\rangle$. Similarly, $Z(S^{\alpha}B)=\langle PQ(S^{\alpha})^{-1}\rangle$. Hence $Z(S^{\alpha}B)=Z(B)$ is equivalent to $\langle P(QS^{\alpha}Q^{-1})\rangle=\langle P\rangle$; this is further equivalent to $\langle P(QS^{\alpha}Q^{-1})\rangle\leqslant\langle P\rangle$, because $b\mapsto b(QS^{\alpha}Q^{-1})$ defines an isomorphism $\langle P\rangle\cong\langle P(QS^{\alpha}Q^{-1})\rangle$, and the two subgroups are finite.

The row vectors $(a_{1},\cdots,a_{m})$ in $\langle P\rangle$ are characterized by the property
$$d_{p}(a_{j})\geqslant k-s_{j},\hspace{5mm} 1\leqslant j\leqslant m.$$
Hence, by considering the $i$-th row of $P(QS^{\alpha}Q^{-1})$,
$\langle P(QS^{\alpha}Q^{-1})\rangle\leqslant\langle P\rangle$ is equivalent to
$$d_{p}(p^{k-s_{i+i_{0}-1}}\cdot(QS^{\alpha}Q^{-1})_{i+i_{0}-1,j})\geqslant k-s_{j}, \hspace{5mm} 1\leqslant i\leqslant m-i_{0}+1,1\leqslant j\leqslant m,$$
which is equivalent to $d_{p}((QS^{\alpha}Q^{-1})_{i,j})\geqslant s_{i}-s_{j}$ when $i>j$ and $i\geqslant i_{0}$; the same inequality holds for free when $j<i<i_{0}$ , since $s_{i}=s_{j}=0$.
\end{proof}

Combining Proposition \ref{prop} and Lemma \ref{lem}, we establish the main result.
\begin{thm}
For $\gamma\in\{1,\cdots,g\}$, choose matrices $Q_{\gamma}\in\textrm{GL}(m,R_{\gamma})$ and $T_{\gamma}\in\textrm{GL}(n_{\gamma},R_{\gamma})$ so that $Q_{\gamma}B_{\gamma}T_{\gamma}$ is
in normal form, $(Q_{\gamma}B_{\gamma}T_{\gamma})_{i,j}=\delta_{i,j}\cdot p_{\gamma}^{s_{\gamma,i}}$. Then $\alpha$ can be lifted
if and only if $d_{p_{\gamma}}((Q_{\gamma}S^{\alpha}Q_{\gamma}^{-1})_{i,j})\geqslant s_{\gamma,i}-s_{\gamma,j}$ for all $\gamma$ and $1\leqslant j<i\leqslant m$.
\end{thm}

\section{Conclusion and example}

Summarizing the last section, we write down the algorithm.

\begin{alg}
\rm Assume the conditions presented in the beginning of Section 3.1.
\begin{enumerate}
\item For $1\leqslant\gamma\leqslant g$, form the matrices $B_{\gamma}$ from the voltages $\phi(L_{i})\in A$.
\item Compute the matrix $S^{\alpha}$ representing the action of $\alpha$ on $H_{1}(\Gamma;\mathbb{Z})$, that is, $\alpha_{\ast}(L_{i})=\sum\limits_{j=1}^{m}(S^{\alpha})_{i,j}L_{j}$.
\item For each $\gamma$, choose matrices $Q_{\gamma}$ and $T_{\gamma}$ to normalize $B_{\gamma}$ so that $(Q_{\gamma}B_{\gamma}T_{\gamma})_{i,j}=\delta_{i,j}\cdot p_{\gamma}^{s_{\gamma,i}}$, with $s_{\gamma,1}\leqslant\cdots\leqslant s_{\gamma,m}$.
\item If $d_{p_{\gamma}}((Q_{\gamma}S^{\alpha}Q_{\gamma}^{-1})_{i,j})\geqslant s_{\gamma,i}-s_{\gamma,j}$, for all $\gamma$ and all $1\leqslant j<i\leqslant m$,
then $\alpha$ can be lifted. Otherwise not.
\end{enumerate}
\end{alg}

We now illustrate this algorithm by an example.

Take $\Gamma$ to be the Petersen graph. The vertices can be labeled by unordered pairs of elements from $X=\{a,b,c,d,e\}$. Let $V=\{0,1,2,3,4,5,6,7,8,9\}$, where $0=\{a,b\}$, $1=\{c,d\}$, $2=\{c,e\}$, $3=\{d,e\}$, $4=\{a,e\}$, $5=\{b,e\}$, $6=\{a,d\}$, $7=\{b,d\}$, $8=\{a,c\}$, $9=\{b,c\}$. The Petersen graph has $V(\Gamma)=V$ and $E(\Gamma)=\{\{i,j\}\colon i,j\in V,i\cap j=\emptyset\}$. It is known that $\textrm{Aut}(\Gamma)\cong S_{5}$, the action being induced from that on $X$.

Fix a spanning tree $T$ as in Figure 1(a), $v_{0}=0$; the induced subgraph is shown in (b). For each cotree edge $\{i,j\}$, choose the arc $(i,j)$ with $i<j$. Let $h_{1}=(5,8),h_{2}=(7,8),h_{3}=(4,7),h_{4}=(4,9),h_{5}=(6,9),h_{6}=(5,6)$.

Take $A=\mathbb{Z}/2\mathbb{Z}\times\mathbb{Z}/2\mathbb{Z}\times\mathbb{Z}/4\mathbb{Z}$.

Define the $T$-reduced voltage assignment $\phi$ by $\phi(h_{1})=(1,1,1)$, $\phi(h_{2})=(1,0,2)$, $\phi(h_{3})=(1,1,2)$, $\phi(h_{4})=(1,0,3)$, $\phi(h_{5})=(1,1,0)$, $\phi(h_{6})=(1,0,0)$.

We shall decide whether the following automorphisms can be lifted:
\begin{align*}
\alpha_{1}&=(0)(2)(13)(67)(49)(58), \\
\alpha_{2}&=(4)(7)(19)(56)(28)(03), \\
\alpha_{3}&=(0)(123)(468)(579),\\
\alpha_{4}&=(0)(1)(2)(3)(45)(67)(89).
\end{align*}
\begin{figure}[h]
  \centering
  \includegraphics[width=0.5\textwidth]{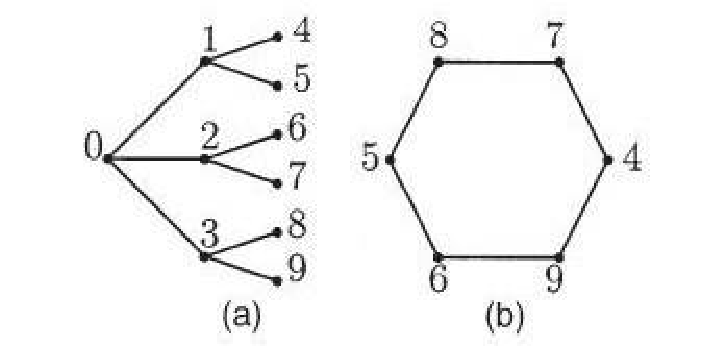}\\
  \caption{(a) The spanning tree $T$; (b) the induced subgraph}\label{1}
\end{figure}

The algorithm is run as follows.

1. We have $m=6$. From $\phi(h_{1}),\cdots,\phi(h_{6})$ we obtain $B=\left(\begin{array}{ccc}  2 & 2 & 1 \\ 2 & 0 & 2 \\ 2 & 2 & 2 \\  2 & 0 & 3 \\  2 & 2 & 0 \\ 2 & 0 & 0 \end{array}\right)$.

2. Determine the matrices $S_{k}=S^{\alpha_{k}}$ for $1\leqslant k\leqslant 4$:\\
$S_{1}=\left(\begin{array}{cccccc}
-1 & 0 & 0 & 0 & 0 & 0\\ 0 & 0 & 0 & 0 & 0 & -1 \\ 0 & 0 & 0 & 0 & -1 & 0 \\ 0 & 0 & 0 & -1 & 0 & 0 \\ 0 & 0 & -1 & 0 & 0 & 0 \\ 0 & -1 & 0 & 0 & 0 & 0 \end{array}\right)$,
$S_{2}=\left(\begin{array}{cccccc}
0 & 0 & 0 & 0 & -1 & 0\\ 0 & -1 & 0 & 0 & 0 & 0 \\ 0 & 1 & 1 & -1 & 0 & 0\\ 0 & 0 & 0 & -1 & 0 & 0\\ -1 & 0 & 0 & 0 & 0 & 0 \\ 1 & 0 & 0 & 0 & -1 & -1
\end{array}\right)$,\\
$S_{3}=\left(\begin{array}{cccccc}
0 & 0 & -1 & 0 & 0 & 0\\ 0 & 0 & 0 & -1 & 0 & 0 \\ 0 & 0 & 0 & 0 & 1 & 0 \\ 0 & 0 & 0 & 0 & 0 & -1 \\ -1 & 0 & 0 & 0 & 0 & 0 \\ 0 & 1 & 0 & 0 & 0 & 0 \end{array}\right)$,
$S_{4}=\left(\begin{array}{cccccc}
0 & 0 & 0 & 1 & 0 & 0\\ 0 & 0 & 0 & 0 & 1 & 0 \\ 0 & 0 & 0 & 0 & 0 & 1 \\ 1 & 0 & 0 & 0 & 0 & 0 \\ 0 & 1 & 0 & 0 & 0 & 0 \\ 0 & 0 & 1 & 0 & 0 & 0 \end{array}\right)$.

3. Normalize $B$ into
$B^{0}=\left(\begin{array}{ccc}  1 & 0 & 0 \\ 0 & 2 & 0 \\ 0 & 0 & 2 \\  0 & 0 & 0 \\  0 & 0 & 0 \\ 0 & 0 & 0 \end{array}\right)$
by $T=\left(\begin{array}{ccc}  0 & 0 & 1 \\ 0 & 1 & 0 \\ 1 & 0 & 0 \end{array}\right)$ and \\
$Q=\left(\begin{array}{cccccc}
3 & 2 & 1 & 0 & 0 & 0\\ 0 & 1 & 1 & 0 & 0 & 0 \\ 2 & 1 & 0 & 0 & 0 & 0 \\ 1 & 1 & 1 & 1 & 0 & 0 \\ 2 & 2 & 1 & 0 & 1 & 0 \\ 2 & 1 & 0 & 0 & 0 & 1 \end{array}\right)$. So
$Q^{-1}=\left(\begin{array}{cccccc}
1 & 3 & 3 & 0 & 0 & 0\\ 2 & 2 & 3 & 0 & 0 & 0 \\ 2 & 3 & 1 & 0 & 0 & 0 \\ 3 & 0 & 1 & 1 & 0 & 0 \\ 0 & 3 & 1 & 0 & 1 & 0 \\ 0 & 0 & 3 & 0 & 0 & 1 \end{array}\right)$,\\
and $s_{1}=0,s_{2}=s_{3}=1,s_{4}=s_{5}=s_{6}=2$.

4. Computing $QS_{k}Q^{-1},1\leqslant k\leqslant 4$, the results are

$QS_{1}Q^{-1}=\left(\begin{array}{cccccc}
1 & 0 & 0 & 0 & 3 & 2\\ 0 & 1 & 0 & 0 & 3 & 3 \\ 2 & 2 & 3 & 0 & 0 & 3 \\ 0 & 2 & 0 & 3 & 3 & 3 \\ 0 & 0 & 0& 0 & 3 & 2 \\ 0 & 0 & 0 & 0 & 0 & 3 \end{array}\right)$,
$QS_{2}Q^{-1}=\left(\begin{array}{cccccc}
1 & 0 & 0 & 1 & 0 & 0\\ 3 & 3 & 0 & 3 & 0 & 0 \\ 2 & 0 & 3 & 0 & 2 & 0 \\ 3 & 2 & 1 & 3 & 1 & 0 \\ 2 & 0 & 0 & 3 & 2 & 0 \\ 3 & 0 & 2 & 0 & 1 & 3 \end{array}\right)$,

$QS_{3}Q^{-1}=\left(\begin{array}{cccccc}
0 & 2 & 0 & 2 & 1 & 0\\ 1 & 3 & 0 & 3 & 1 & 0 \\ 1 & 2 & 1 & 3 & 0 & 0 \\ 3 & 0 & 0 & 3 & 1 & 3 \\ 1 & 2 & 0 & 2 & 1 & 0 \\ 3 & 0 & 0 & 3 & 0 & 0 \end{array}\right)$,
$QS_{4}Q^{-1}=\left(\begin{array}{cccccc}
1 & 2 & 0 & 3 & 2 & 3\\ 0 & 3 & 0 & 0 & 1 & 1 \\ 2 & 3 & 1 & 2 & 1 & 0 \\ 0 & 2 & 0 & 1 & 1 & 1 \\ 0 & 0 & 2 & 2 & 2 & 1 \\ 0 & 2 & 0 & 2 & 1 & 0 \end{array}\right)$.

For $k=1$ or $4$, the matrix $QS_{k}Q^{-1}$ satisfies $d_{p}((QS_{k}Q^{-1})_{i,j})\geqslant s_{i}-s_{j}$ for all $1\leqslant j<i\leqslant 6$.
But for $k=2$ or $3$, it does not; for instance, $(QS_{2}Q^{-1})_{2,1}=3,(QS_{3}Q^{-1})_{2,1}=1$.

We conclude that $\alpha_{1},\alpha_{4}$ can be lifted, while $\alpha_{2},\alpha_{3}$ cannot!




\begin{thebibliography}{}


\bibitem{Crite}
S. F. Du, J. H. Kwak, M. Y. Xu,
Linear criteria for lifting automorphisms of elementary abelian regular coverings,
Linear Algebra and Its Applications 373 (2003) 101-119.

\bibitem{Voltage}
J. L. Gross, T. W. Tucker,
Generating all graph coverings by permutation voltage assignments,
Discrete Mathematics 18 (1977) 273-283.

\bibitem{Graph}
J. L. Gross, T. W. Tucker,
Topological Graph Theory,
Wiley-Interscience, New York, 1987.

\bibitem{AT}
A. Hatcher,
Algebraic Topology,
Cambridge University Press, Cambridge, 2002.

\bibitem{K5}
B. Kuzman,
Arc-transitive elementary abelian covers of the complete graph $K_{5}$,
Linear Algebra and Its Applications 433 (2010) 1909-1921.


\bibitem{Lift}
A. Malni$\check{\textup{c}}$,
Group actions, coverings and lifts of automorphisms,
Discrete Mathematics 182 (1998) 203-218.

\bibitem{Elem}
A. Malni$\check{\textup{c}}$, D. Maru$\check{\textup{s}}$i$\check{\textup{c}}$, P. Poto$\check{\textup{c}}$nik,
Elementary abelian covers of graphs,
Journal of Algebraic Combinatorics 20 (2004) 71-97.

\bibitem{Invar}
A. Malni$\check{\textup{c}}$, P. Poto$\check{\textup{c}}$nik,
Invariants subspaces, duality, and covers of the Petersen graph,
European Journal of Combinatorics 27 (2006) 971-989.


\bibitem{Pappus}
J-M. Oh,
Arc-transitive elementary abelian covers of the Pappus graph,
Discrete Mathematics 309 (2009) 6590-6611.

\end{thebibliography}
\end{document}